\documentclass[12pt,leqno,oneside]{amsart}
\usepackage[T1]{fontenc}
\usepackage{amsmath,amssymb,amsthm,mathtools}
\usepackage{charter}
\usepackage[width=6.4in,height=8.5in]{geometry}
\usepackage[colorlinks=true,linkcolor=blue,citecolor=blue,urlcolor=blue]{hyperref}
\newtheorem{theorem}{Theorem}[section]
\newtheorem{proposition}[theorem]{Proposition}
\newtheorem{lemma}[theorem]{Lemma}
\newtheorem{corollary}[theorem]{Corollary}
\theoremstyle{definition}

\theoremstyle{remark}
\newtheorem{remark}[theorem]{Remark}
\numberwithin{equation}{section}
\newcommand{\F}{\mathcal F}
\newcommand{\C}{\mathbb C}
\newcommand{\E}{\mathbb E}
\newcommand{\Prob}{\mathbb P}
\newcommand{\ddbar}{\partial\overline\partial}
\newcommand{\norm}[1]{\left\|#1\right\|}
\DeclareMathOperator{\tr}{tr}
\DeclareMathOperator{\Vol}{Vol}
\title[Fekete points and random sections]{Fekete Points as Norming Sets for Random Holomorphic Sections}
\author{Turgay Bayraktar}
\thanks{T. Bayraktar is partially supported by T\"{U}B\.{I}TAK grant ARDEB-1001/124F370}
\address{Faculty of Engineering and Natural Sciences, Sabanc{\i} University, \.Istanbul, Turkey}
\email{tbayraktar@sabanciuniv.edu}
\keywords{Fekete points, norming sets, positive line bundles, random holomorphic sections}
\subjclass[2020]{Primary 32L05; Secondary 32A25, 41A05, 60G15, 94A20}
\date{October 22, 2026}
\begin{document}
\begin{abstract}
We prove that Fekete configurations associated with a positive Hermitian line
bundle over a compact K\"ahler manifold are asymptotically exact norming sets
for random holomorphic sections. For every fixed $a>0$, the ratio of the
supremum norm of a degree-$k$ section to its maximum on the degree-$\lceil
ak\rceil$ Fekete configuration converges to one in probability. This holds
for any background K\"ahler form: Fekete equidistribution provides enough
sampling points near the maximum of the limiting pointwise variance.
We obtain polynomial upper bounds for the tail of the ratio and
almost-sure norm comparisons. In contrast, when the limiting variance is
nonconstant, we construct separated arrays of cardinality comparable to
the section-space dimension for which the ratio converges to a constant
strictly greater than one. General separated arrays still satisfy
probabilistic norming inequalities with degree-independent constants, and
every such array is asymptotically exact if and only if the volume form is
proportional to $c_1(L,h)^n$. We also prove a discrete $L^2$ norm comparison with exponentially
small exceptional probability. The norm comparisons and convergence
statements hold in both the Gaussian and spherical ensembles.
\end{abstract}
\maketitle

\section{Introduction}
Let $(X,\omega)$ be a compact K\"ahler manifold of complex dimension $n\geq1$, and let $(L,h)\to X$ be a holomorphic line bundle with a smooth positively curved Hermitian metric. We put
\[
H_k=H^0(X,L^k),\qquad N_k=\dim H_k,\qquad L^k=L^{\otimes k}.
\]
All norms of sections of $L^k$ are computed with respect to the metric $h^k=h^{\otimes k}$ induced by $h$. We also denote the distance by $d$ that is induced by the fixed background metric $\omega$. We use only sufficiently large $k$, so that $L^k$ is globally generated.

For a basis $S_1^k,\ldots,S_{N_k}^k$ of $H_k$, the Vandermonde-type determinant
\[
\operatorname{vdm}_k(x_1,\ldots,x_{N_k})
=\det\bigl(S_i^k(x_j)\bigr)_{i,j=1}^{N_k}
\]
is a section of $(L^k)^{\boxtimes N_k}$ over $X^{N_k}$. A \emph{Fekete configuration} $\F_k=\{x_1,x_2,\dots, x_{N_k}\}$ maximizes its pointwise norm. Such configurations exist by compactness, are independent of the basis, and need not be unique. We make an arbitrary choice at each level $k$.

After choosing a unit vector $u_j\in L^k_{x_j}$ at each point of $\F_k$, let $\ell_j^k\in H_k$ be the \textit{Lagrange sections} satisfying
\[
\ell_j^k(x_i)=\delta_{ij}u_i.
\]
Fekete extremality gives $\norm{\ell_j^k}_\infty=1$ (see \S \ref{sec:deterministic} for details). The Lebesgue constant is defined by
\begin{equation}\label{eq:lebesgue}
\Lambda_k=\max_{x\in X}\sum_{j=1}^{N_k}|\ell_j^k(x)|_{h^k}.
\end{equation}
It is independent of the chosen unit vectors. Note that $\Lambda_k$ equals both the norm of the interpolation projection on continuous sections and the best constant in
\begin{equation}\label{eq:critical-norming}
\norm{s}_\infty\leq\Lambda_k\max_{\lambda\in\F_k}|s(\lambda)|_{h^k},\qquad s\in H_k.
\end{equation}
In particular, $1\leq\Lambda_k\leq N_k$. These assertions are verified in Section~\ref{sec:deterministic}.

For $a>0$, define the \emph{optimal norming constant}
\begin{equation}\label{eq:oversampled-constant}
\Lambda_k(a)=\inf\left\{C\geq0:\norm{s}_\infty\leq C
\max_{\lambda\in\F_{\lceil ak\rceil}}|s(\lambda)|_{h^k}
\text{ for every }s\in H_k\right\},
\end{equation}
with the value $+\infty$ if no such constant exists. In particular, $\Lambda_k(1)=\Lambda_k$. When $a\ne1$, this is a norming constant, not the norm of a Lagrange interpolation operator on $H_k$.

Sampling and interpolation of Fekete arrays have been studied in several geometric settings; see, for example, \cite{Bernt03,BBWN,LOC,AOC,MOC,Bos} and the references therein. We record the following deterministic threshold as background for the random results: 
\begin{theorem}\label{prop:deterministic}
Let $(L,h)\to(X,\omega)$ be as above, and choose a separation constant $0<\delta\leq1$ as in Proposition~\ref{prop:sep}.
\begin{enumerate}
\item[(i)] There is a constant $C=C(X,\omega,L,h)>0$, independent of $a$ and $k$, such that for every fixed $a>1$ and $k\geq k_0(a)$,
\[
\Lambda_k(a)\leq C\,\delta^{-2n}\left(\frac{a}{a-1}\right)^n.
\]
\item[(ii)] $\Lambda_k=\Lambda_k(1)\longrightarrow+\infty$ as $k\to\infty$.
\end{enumerate}
Consequently, for each fixed $a>0$, the array $\{\F_{\lceil ak\rceil}\}_k$ is $L^\infty$-sampling for $H_k$ if and only if $a>1$.
\end{theorem}
For the proof, part (i), including the exponent $n$ in the oversampling parameter, follows from the proof of Lemma 6 in \cite{LOC}. Although that lemma is stated for $L^2$-sampling, its proof also yields the required $L^\infty$ estimate. Part (ii) follows by combining the local counting estimates of \cite{LOC} with the strict interpolation-density theorem of \cite{GHOR} for the Bargmann–Fock space. While the underlying density obstruction is well understood, we are not aware of an explicit statement in the literature asserting the divergence of the Lebesgue constants at the critical level. 

Here $L^\infty$-sampling means that the norming inequality holds for all sections and all sufficiently large $k$, with a constant independent of $k$. For every fixed $a>0$, separation at level $\lceil ak\rceil$ is equivalent to separation at scale $k^{-1/2}$, with a constant depending on $a$. Determining the precise growth rate of these constants remains an interesting question. In the case of  $\mathbb P^1$ endowed with the Fubini--Study metric, the rate $\Lambda_k\gtrsim\sqrt{N_k}$ is given in Proposition~\ref{prop:P1-rate}.

We next endow $H_k$ with the normalized complex Gaussian measure $\Prob_k$ determined by the $L^2$ inner product associated with $(X,\omega,L,h)$. In an orthonormal basis, the coefficients are independent centered complex Gaussians of variance $N_k^{-1}$ (see \S \ref{sec:random} for details). We first state a probabilistic norm comparison that extends to general separated arrays, and then give the stronger conclusion supplied by Fekete equidistribution \cite{BBWN}.

\begin{theorem}\label{thm:random}
For every $q>0$ there are constants $C_q>0$ and $k_0(q)$, with $C_q=O(\sqrt{1+q})$, such that, for $k\geq k_0(q)$,
\begin{equation}\label{eq:random-main}
\Prob_k\left\{s\in H_k:\norm{s}_\infty>
C_q\max_{\lambda\in\F_k}|s(\lambda)|_{h^k}\right\}=O(k^{-q}).
\end{equation}
As a result, for any fixed $q>1$ almost every sequence $(s_k)$ under the product measure satisfies the corresponding norming inequality for all sufficiently large $k$, with the deterministic constant $C_q$.
\end{theorem}

The proof uses a large subset of evaluations with a uniformly positive covariance matrix, which gives a lower-tail estimate stronger than every polynomial bound. A separate Gaussian upper-tail estimate controls the full supremum. In fact, the argument proves the same result for any deterministic array $\Gamma_k$ satisfying
\begin{equation}\label{eq:general-array}
 d(\lambda,\mu)\geq\delta_0 k^{-1/2}\quad(\lambda\ne\mu),
 \qquad \#\Gamma_k\geq c_0 k^n,
\end{equation}
where $\delta_0,c_0>0$ are fixed. Note that separation also gives $\#\Gamma_k=O(k^n)$. Thus the random norming property does not require Fekete extremality.

The additional distribution properties of Fekete arrays yield a stronger
conclusion for an arbitrary background K\"ahler form. Set
\[
 V_L=\frac1{n!}\int_Xc_1(L,h)^n,\qquad
 b_0(x)=\frac{c_1(L,h)^n}{\omega^n}(x).
\]
\[
 \sigma_{\max}^2=\frac{\max_Xb_0}{V_L},\qquad
 \Theta_k^{\max}=\sigma_{\max}\sqrt{n\log k}.
\]
Here, the function $b_0/V_L$ is the limiting pointwise variance of our Gaussian
ensemble.

\begin{theorem}[Asymptotically exact Fekete norming]\label{thm:exact-intro}
For every fixed $a>0$, let $\Gamma_k=\F_{\lceil ak\rceil}$ and define
$R_k(s)=\norm{s}_\infty/\max_{\lambda\in\Gamma_k}|s(\lambda)|$.
Then
\[
 \frac{\norm{s}_\infty}{\Theta_k^{\max}}\longrightarrow1,\qquad
 \frac{\max_{\Gamma_k}|s|}{\Theta_k^{\max}}\longrightarrow1,\qquad
 R_k(s)\longrightarrow1
 \quad\text{in probability}.
\]
These conclusions hold in both the Gaussian and spherical ensembles,
without any compatibility assumption between $\omega$ and $c_1(L,h)$.
Moreover, for every $q>0$ and every $C>\sqrt{1+q/n}$,
\[
 \Prob_k\{R_k>C\}=O(k^{-q}).
\]
The ratio is defined to be one at the zero section and $+\infty$ for
nonzero sections whose sampled values all vanish.
\end{theorem}

Theorem~\ref{thm:fekete-exact} gives the amplitude tail bounds and the
almost-sure consequence. At $a=1$, the theorem contrasts with the divergence
of the deterministic Lebesgue constants; it also applies when $a<1$,
where deterministic norming fails by dimension counting.
The essential additional input is Fekete equidistribution
\cite{BBWN}: every fixed neighborhood of a maximum of $b_0$ contains a
positive proportion of the nodes.

This input cannot be replaced by separation and cardinality alone.
If $b_0$ is nonconstant, Proposition~\ref{prop:variance-counterexample}
constructs arrays satisfying \eqref{eq:general-array} for which
\[
 R_k\longrightarrow\sqrt{\frac{\max_Xb_0}{b'}}>1
 \quad\text{in probability}
\]
for some $0<b'<\max_Xb_0$. Conversely, when $b_0$ is constant,
Theorem~\ref{thm:sharp} shows that every array satisfying
\eqref{eq:general-array} is asymptotically exact. Thus every such array
is asymptotically exact if and only if $b_0$ is constant
(Corollary~\ref{cor:variance-dichotomy}), whereas Fekete arrays are
asymptotically exact for every background K\"ahler form.
The general-array results use separation; the Fekete theorem additionally
uses equidistribution.

We also obtain an $L^2$ comparison with exceptional probability $O(e^{-ck^n})$; see Theorem~\ref{thm:L2}. Again this is a statement about typical sections, not a deterministic sampling inequality or a uniformly stable interpolation theorem. The distinction is essential: unique interpolation at a Fekete configuration does not imply uniform stability of the interpolation operators.

A related deterministic problem was studied by Aguilar-Hern\'andez \cite{AH19}. The \cite[Theorem~1.1]{AH19} characterizes measurable sets $G_k\subset X$ satisfying $\norm{s}_2^2\leq C\int_{G_k}|s|^2\,dV$ for every section by relative volume density in every ball of radius $R/\sqrt{k}$, for some fixed $R>0$. Moreover \cite[Theorem~4.2]{AH19} characterizes uniform upper $L^2$ bounds for measures, and Example~4.5 gives the discrete Plancherel--P\'olya inequality for separated arrays used below. Our results concern discrete norm comparisons for Gaussian sections outside exceptional sets with quantitative probability bounds. The hypotheses \eqref{eq:general-array} impose separation and total cardinality, without requiring relative density in every ball at scale $k^{-1/2}$.

The paper is organized as follows: Section~\ref{sec:prelim} contains the background on complex geometry and analytic estimates. Section~\ref{sec:deterministic} proves Theorem~\ref{prop:deterministic}. Section~\ref{sec:random} proves the probabilistic supremum-norm and $L^2$ results.

\section{Geometric and analytic preliminaries}\label{sec:prelim}
Throughout, $c,C>0$ denote constants that may change from line to line; unless specified otherwise, they depend only on the fixed geometry. Constants with designated subscripts retain their stated meanings.
For a fixed local holomorphic frame $e_\alpha$ we write $|e_\alpha|_h^2=e^{-\phi_\alpha}$. If $e_\beta=g_{\alpha\beta}e_\alpha$, then
\[
\phi_\alpha=\phi_\beta+\log|g_{\alpha\beta}|^2.
\]
We use the convention
\[
c_1(L,h)=\frac{i}{2\pi}\ddbar\phi_\alpha>0.
\]
A section of $L^k$ has local expression $s=f_\alpha e_\alpha^{\otimes k}$, with $f_\alpha=g_{\alpha\beta}^k f_\beta$ and $|s|_{h^k}=|f_\alpha|e^{-k\phi_\alpha/2}$.
Set
\[
V_L=\frac1{n!}\int_X c_1(L,h)^n>0,\qquad dV=\frac{\omega^n}{n!}.
\]
The Riemann--Roch theorem and positivity of $(L,h)$ give
\begin{equation}\label{eq:dimension}
N_k=V_L k^n+O(k^{n-1}).
\end{equation}
Note that no normalization $V_L=1$ is imposed. One can define the $L^2$ inner product on $H^0(X,L^k)$ via
\begin{equation}
\langle s,t\rangle=\int_X\langle s(x),t(x)\rangle_{h^k}\,dV(x).
\end{equation}

\subsection{Bergman kernels and peak sections}
Let $S_1^k,\ldots,S_{N_k}^k$ be an orthonormal basis. The Bergman kernel is the smooth bundle map
\[
K_k(x,y):L_y^k\longrightarrow L_x^k,\qquad
K_k(x,y)v=\sum_{j=1}^{N_k}S_j^k(x)\langle v,S_j^k(y)\rangle_{h^k},
\]
where the Hermitian product is linear in its first variable. On the diagonal it is given by $K_k(x,x)=B_k(x)\operatorname{Id}$, where
\[
B_k(x)=\sum_{j=1}^{N_k}|S_j^k(x)|_{h^k}^2,\qquad \int_X B_k\,dV=N_k.
\]
The diagonal expansion (see \cite{Cat,Tia,Zel,MM07}) reads
\begin{equation}\label{eq:diagonal}
B_k(x)=b_0(x)k^n+O(k^{n-1}),\qquad
b_0(x)=\frac{c_1(L,h)^n}{\omega^n}(x)>0,
\end{equation}
uniformly on $X$. 
Off-diagonal asymptotics of Bergman kernels have been studied by many authors in
various settings; see, e.g., \cite{SZ08, BCM} and the references therein. We use
only the following exponential decay estimate (see \cite[Theorem~1]{MM15}): there
are constants $C,c>0$ such that, for all sufficiently large $k$ and all $x,y\in X$,
\begin{equation}\label{eq:off-diagonal}
  |K_k(x,y)| \le C\,k^n e^{-c\sqrt{k}\,d(x,y)} .
\end{equation}

For a unit vector $u\in L_x^r$, define
\[
p_{r,x}(y)=\frac{K_r(y,x)u}{B_r(x)}\in H^0(X,L^r).
\]
Equations~\eqref{eq:diagonal}--\eqref{eq:off-diagonal} imply, uniformly in $x$ and all sufficiently large $r$,
\begin{equation}\label{eq:peak}
|p_{r,x}(x)|=1,\quad
|p_{r,x}(y)|\leq C e^{-c\sqrt r\,d(x,y)},\quad
\norm{p_{r,x}}_2^2=B_r(x)^{-1}\leq Cr^{-n}.
\end{equation}

\subsection{Local estimates and packing}
We will use the following standard local estimates (cf. \cite[Lemma 2]{LOC}):
\begin{lemma}\label{lem:local}
For all sufficiently large $k$ and $s\in H_k$,
\begin{equation}\label{eq:local-submean}
|s(x)|_{h^k}^2\leq C k^n\int_{B(x,k^{-1/2})}|s(y)|_{h^k}^2\,dV(y).
\end{equation}
There is also a constant $C$ such that
\begin{equation}\label{eq:bernstein}
\bigl||s(x)|_{h^k}-|s(y)|_{h^k}\bigr|
\leq C\sqrt{k}\,d(x,y)\norm{s}_\infty.
\end{equation}
If $\Gamma_k$ is $\delta_0 k^{-1/2}$-separated, then
\begin{equation}\label{eq:PP}
k^{-n}\sum_{\lambda\in\Gamma_k}|s(\lambda)|_{h^k}^2
\leq C_{\delta_0}\norm{s}_2^2.
\end{equation}
For \eqref{eq:PP}, see also \cite[Theorem 4.2 and Example 4.5]{AH19}.
\end{lemma}
\begin{proof}
At a given point, change the holomorphic frame so that the local weight has zero constant and linear terms there. On a ball of radius $k^{-1/2}$, the factor $e^{-k\phi}$ is then bounded above and below by uniform constants. The submean inequality and Cauchy estimates in these frames give \eqref{eq:local-submean} and the bound $|\nabla s|\leq C\sqrt{k}\norm{s}_\infty$ for the Chern covariant derivative. Integration along a minimizing geodesic gives \eqref{eq:bernstein}, also across zeros by continuity. Summing \eqref{eq:local-submean} and using bounded overlap of the balls centered at a separated array proves \eqref{eq:PP}.
\end{proof}

\begin{lemma}\label{lem:packing}
Suppose $\Gamma$ is $\rho m^{-1/2}$-separated, where $0<\rho\leq1$. There is a geometric constant $C$ such that
\begin{equation}\label{eq:packing}
\#(\Gamma\cap B(x,t))\leq C\rho^{-2n}(1+\sqrt m\,t)^{2n}.
\end{equation}
For each fixed $c>0$, if $1\leq r\leq m$, then
\begin{equation}\label{eq:packing-exp}
\sum_{\lambda\in\Gamma}e^{-c\sqrt r\,d(x,\lambda)}
\leq C_c\rho^{-2n}(m/r)^n.
\end{equation}
\end{lemma}
\begin{proof}
Disjoint balls of radius $\rho/(3\sqrt m)$ around the relevant points lie in $B(x,t+\rho/(3\sqrt m))$. Uniform small-ball volume estimates on the compact manifold give \eqref{eq:packing}, with the large-radius cases absorbed by compactness. Put $A(t)=\#(\Gamma\cap B(x,t))$. The identity
\[
e^{-c\sqrt r\,d(x,\lambda)}
=\int_{d(x,\lambda)}^\infty c\sqrt r\,e^{-c\sqrt r\,t}\,dt
\]
and Tonelli's theorem give an upper bound by
\[
C\rho^{-2n}\int_0^\infty e^{-u}
\left(1+\frac{\sqrt m}{c\sqrt r}u\right)^{2n}\,du
\leq C_c\rho^{-2n}(m/r)^n.
\]
\end{proof}

\section{Deterministic norming and the critical level}\label{sec:deterministic}
\subsection{Interpolation and separation}
For all sufficiently large $k$, there is a configuration at which the determinant of evaluations is nonzero: the evaluation functionals span $H_k^*$, since a section vanishing at every point is zero. A maximizing configuration therefore has nonzero determinant and is unisolvent.

For a unisolvent configuration, the intrinsic Lagrange section associated with $x_j$ is the determinant ratio
\begin{equation}
\frac{\operatorname{vdm}_k(x_1,\ldots,x_{j-1},x,x_{j+1},\ldots,x_{N_k})}
{\operatorname{vdm}_k(x_1,\ldots,x_{N_k})},
\end{equation}
viewed as a holomorphic section of $L^k\otimes(L^k_{x_j})^*$. The numerator and denominator here are determinant sections, not their absolute values. Applying this bundle map to the chosen unit vector $u_j$ yields $\ell_j^k$. Fekete maximality gives $|\ell_j^k(x)|\leq1$ for all $x$ and equality at $x_j$.

If $s(x_j)=a_ju_j$, its interpolant is $\pi_k(s)=\sum_j a_j\ell_j^k$. Therefore $\norm{\pi_k}\leq\Lambda_k$. Conversely, at a maximizer $x$ of the sum in \eqref{eq:lebesgue}, choose $|a_j|=1$ so that the nonzero vectors $a_j\ell_j^k(x)$ have a common phase in $L_x^k$. The interpolating holomorphic section has sample maximum one and value of modulus $\Lambda_k$ at $x$. Such data also extend to a continuous section of supremum norm one, using disjoint neighborhoods of the nodes. Hence
\begin{equation}\label{eq:optimal}
\norm{\pi_k}=\Lambda_k=
\sup_{s\in H_k\setminus\{0\}}
\frac{\norm{s}_\infty}{\max_{\lambda\in\F_k}|s(\lambda)|}.
\end{equation}
As usual, for every continuous section $s$,
\[
\norm{s-\pi_k(s)}_\infty\leq(1+\Lambda_k)
\inf_{t\in H_k}\norm{s-t}_\infty.
\]

Next, we show that Fekete configurations are well separated (cf. \cite[Lemma 3]{LOC}): 
\begin{proposition}\label{prop:sep}
There are $k_*$ and $0<\delta\leq1$, depending only on the fixed geometry, such that every Fekete configuration satisfies
\[
d(\lambda,\mu)\geq\delta/\sqrt{k}\qquad
(\lambda,\mu\in\F_k,\ \lambda\ne\mu,\ k\geq k_*).
\]
\end{proposition}
\begin{proof}
For distinct nodes $x_i,x_j$, the Lagrange section $\ell_j^k$ has $\norm{\ell_j^k}_\infty=1$, $|\ell_j^k(x_j)|=1$, and $\ell_j^k(x_i)=0$. Hence \eqref{eq:bernstein} gives
\[
1=\bigl||\ell_j^k(x_j)|-|\ell_j^k(x_i)|\bigr|
\leq C\sqrt{k}\,d(x_i,x_j).
\]
Take $\delta=\min\{1,C^{-1}\}$. 
\end{proof}

\subsection{Quantitative oversampling}
The following is a version of the peak-section argument in \cite[Section 4.3, proof of Lemma 6]{LOC}. Their estimate~(20) already supplies the power $n$ in the $L^\infty$ bound; we give the argument with the separation constant displayed.
\begin{proof}[Proof of Theorem~\ref{prop:deterministic}(i)]
Fix $a>1$, put $m=\lceil ak\rceil$ and $r=m-k$, and take $k$ large enough that \eqref{eq:peak} applies at degree $r$ and Proposition~\ref{prop:sep} applies at degree $m$. For $s\in H_k$, choose $x$ with $|s(x)|=\norm{s}_\infty$. The section $s\otimes p_{r,x}$ belongs to $H_m$, with no rounding ambiguity. Interpolating at $\F_m$ and using the degree-$m$ Lagrange bound gives
\begin{align*}
\norm{s}_\infty
&=|(s\otimes p_{r,x})(x)|\\
&\leq\sum_{\lambda\in\F_m}|s(\lambda)|\,|p_{r,x}(\lambda)|\\
&\leq C\max_{\lambda\in\F_m}|s(\lambda)|
\sum_{\lambda\in\F_m}e^{-c\sqrt r\,d(x,\lambda)}\\
&\leq C\delta^{-2n}(m/r)^n
\max_{\lambda\in\F_m}|s(\lambda)|,
\end{align*}
where the last line is Lemma~\ref{lem:packing}, applied with $\rho=\delta$. Since $m\geq ak$ and $t\mapsto t/(t-k)$ is decreasing for $t>k$,
\[
\frac{m}{m-k}\leq\frac{a}{a-1}.
\]
All constants in the argument are independent of $a$ and $k$; only the threshold ensuring $r$ is sufficiently large depends on $a$.
\end{proof}

\subsection{Local density and a Bargmann--Fock limit}
For the next argument only, it is convenient to use the auxiliary K\"ahler metric with K\"ahler form $(i/2)\ddbar\phi$. This does not change the Fekete configurations or their Lebesgue constants. We also denote the induced distance by $d_h$.

We recall the local discrepancy estimate of \cite[Theorem 1 and Section 7.1, equation (40)]{LOC}. With $\mu_L=c_1(L,h)^n/n!$, its consequence in rescaled balls is
\begin{equation}\label{eq:local-count}
\#\bigl(\F_k\cap B_h(x,R/\sqrt k)\bigr)
=\frac{N_k}{V_L}\int_{B_h(x,R/\sqrt k)}\mu_L
+O(R^{2n-1}).
\end{equation}
We use this estimate as follows: for each fixed $R\geq R_0$, it holds for all sufficiently large $k$ (depending on $R$), with an error bounded by $C R^{2n-1}$, where $C$ is independent of $x$ and $R$. In the notation of \cite[Theorem 1]{LOC}, the physical radius is $r=R/\sqrt{k}$; their proof in Section~7.1, equation~(40), gives precisely the needed rescaled discrepancy. Equivalently, multiplying the relative error $O((r\sqrt{k})^{-1})=O(R^{-1})$ by the main term $O(R^{2n})$ gives the displayed error. Below we first let $k\to\infty$ at fixed $R$, and then let $R\to\infty$. This uses shrinking physical balls, not merely equidistribution for a fixed physical radius. Weak equidistribution itself was proved in greater generality in \cite{BBWN}.

At any fixed $x_0$, choose holomorphic coordinates $z$ and a holomorphic frame $e$ so that
\[
z(x_0)=0,\qquad |e|_h^2=e^{-\phi},\qquad
\phi(z)=|z|^2+O(|z|^3).
\]
The metric $d_h$ is Euclidean to leading order in these coordinates. If $U$ is the coordinate neighborhood, write
\[
\Sigma_k=\{\sqrt k\,z(\lambda):\lambda\in\F_k\cap U\}.
\]

After passing to a subsequence $k_j\to\infty$, the
rescaled configurations $\Sigma_{k_j}$ converge locally
to a uniformly separated set $\Sigma\subset\mathbb C^n$.
More precisely, the separation and packing estimates
allow us to choose the subsequence so that the counting
measures
\[
  \sum_{\xi\in\Sigma_{k_j}}\delta_\xi
  \longrightarrow
  \sum_{\sigma\in\Sigma}\delta_\sigma
\]
converge weakly on $\mathbb C^n$. In particular, for
every bounded domain $V\Subset\mathbb C^n$ with
$\partial V\cap\Sigma=\varnothing$, the sets
$\Sigma_{k_j}\cap V$ and $\Sigma\cap V$ have the same
cardinality for all sufficiently large $j$. Their
points can then be labeled so that each
$\sigma\in\Sigma\cap V$ corresponds to a unique
$\sigma_j\in\Sigma_{k_j}\cap V$ satisfying
$\sigma_j\to\sigma$. We will use this local matching
both in the density calculation and in the
interpolation argument below.

\begin{lemma}\label{lem:critical-density}
Every such local limit has critical Bargmann--Fock density:
\[
D^+(\Sigma)=D^-(\Sigma)=1,
\]
where
\[
D^+(\Sigma)=\limsup_{R\to\infty}\sup_{u\in\C^n}
\frac{\#(\Sigma\cap B(u,R))}{\pi^{-n}\Vol(B(u,R))}
\]
and
\[
D^-(\Sigma)=\liminf_{R\to\infty}\inf_{u\in\C^n}
\frac{\#(\Sigma\cap B(u,R))}{\pi^{-n}\Vol(B(u,R))}.
\]
\end{lemma}

\begin{proof}
In the chosen coordinates,
\[
  \mu_L=\pi^{-n}(1+O(|z|))\,d\mathcal L(z),
\]
where $\mathcal L$ denotes Euclidean volume. Fix
$u\in\mathbb C^n$ and $R>R_0+1$, and put
\[
  x_j(u)=z^{-1}(u/\sqrt{k_j}).
\]
Since the rescaled metric converges locally uniformly
to the Euclidean metric, for every fixed
$\eta\in(0,1)$ and all sufficiently large $j$,
\[
\begin{aligned}
 B_h\left(x_j(u),\frac{R-\eta}{\sqrt{k_j}}\right)
 &\subset
 \left\{x\in U:\sqrt{k_j}z(x)\in B(u,R)\right\}\\
 &\subset
 B_h\left(x_j(u),\frac{R+\eta}{\sqrt{k_j}}\right).
\end{aligned}
\]
The local discrepancy estimate (3.3), together with
$N_{k_j}/V_L=k_j^n+O(k_j^{n-1})$, therefore gives
\[
\begin{aligned}
 \frac{(R-\eta)^{2n}}{n!}
   -C(R-\eta)^{2n-1}+o(1)
 &\le \#(\Sigma_{k_j}\cap B(u,R))\\
 &\le \frac{(R+\eta)^{2n}}{n!}
   +C(R+\eta)^{2n-1}+o(1).
\end{aligned}
\]
Here $C$ is independent of $u$, $R$, and $\eta$,
and the limits are taken with $u$, $R$, and $\eta$
fixed.

Suppose first that $\partial B(u,R)\cap\Sigma
=\varnothing$. By the local convergence described
above,
\[
  \#(\Sigma_{k_j}\cap B(u,R))
  =\#(\Sigma\cap B(u,R))
\]
for all sufficiently large $j$. Letting $j\to\infty$
and then $\eta\downarrow0$ yields
\[
  \left|
    \#(\Sigma\cap B(u,R))-\frac{R^{2n}}{n!}
  \right|
  \le CR^{2n-1}.
\]
The same estimate for arbitrary $R$ follows by
approximating $R$ from above and below by radii
whose boundary spheres contain no points of
$\Sigma$.

The constant $C$ is independent of $u$, because
the original discrepancy estimate is uniform in
its center. Although $u$ is fixed when passing to
the local limit, the resulting inequality therefore
holds for every $u\in\mathbb C^n$ with the same
constant.

Since
\[
  \pi^{-n}\operatorname{Vol}(B(u,R))
  =\frac{R^{2n}}{n!},
\]
dividing the preceding estimate by this quantity
and taking the supremum or infimum over $u$,
followed by $R\to\infty$, proves
\[
  D^+(\Sigma)=D^-(\Sigma)=1.
\]
\end{proof}

Let $\mathcal B^\infty$ denote the Bargmann--Fock space with norm
\[
\norm{f}_{\mathcal B^\infty}=
\sup_{z\in\C^n}|f(z)|e^{-|z|^2/2}.
\]
The density normalization in Lemma~\ref{lem:critical-density} is the one associated with this weight. The strict density theorem \cite[Theorem 1.2(b)]{GHOR}, including its $p=\infty$ case, states that
\begin{equation}\label{eq:strict-density}
\Sigma\text{ interpolating for }\mathcal B^\infty
\quad\Longrightarrow\quad D^+(\Sigma)<1.
\end{equation}
Here interpolation means that every bounded family of normalized data admits an interpolant with a uniform norm bound. Explicitly, for bounded $(d_\sigma)_{\sigma\in\Sigma}$ one requires
\[
f(\sigma)e^{-|\sigma|^2/2}=d_\sigma\quad(\sigma\in\Sigma),\qquad
\norm{f}_{\mathcal B^\infty}\leq C\sup_{\sigma\in\Sigma}|d_\sigma|.
\]

\begin{remark}
Theorem~5 of \cite{LOC} rules out simultaneous $L^2$ sampling and interpolation. It does not directly give the assertion about $\Lambda_k$, which concerns $L^\infty$ norms. Moreover, the passage from $L^\infty$ sampling to $L^2$ sampling in \cite[Proposition 1]{LOC} uses a positive oversampling factor, so it does not by itself settle the critical level. We therefore give the short local-limit argument below, using the $p=\infty$ case of \cite[Theorem 1.2(b)]{GHOR}. The argument excludes every bounded subsequence of $\Lambda_k$, proving divergence rather than merely the failure of a uniform bound.
\end{remark}

\begin{proof}[Proof of Theorem~\ref{prop:deterministic}(ii)]
Suppose, to the contrary, that $\Lambda_{k_j}\le M$
along a sequence $k_j\to\infty$. By the interpolation
identity (3.2), arbitrary data of modulus at most
one prescribed in unit frames at the points of
$\mathcal F_{k_j}$ admit an interpolating section
$s_j\in H_{k_j}$ satisfying
\[
  \|s_j\|_\infty\le M.
\]
After passing to a subsequence, let
$\Sigma_{k_j}\to\Sigma$ be a local limit as above.

We show that $\Sigma$ is interpolating for the
Bargmann--Fock space $\mathcal B^\infty$, with
interpolation constant at most $M$.

First, fix a finite subset $E\subset\Sigma$ and
normalized interpolation data
$(d_\sigma)_{\sigma\in E}$ with
$|d_\sigma|\le1$. Choose $R>0$ such that
\[
  E\subset B(0,R),
  \qquad
  \partial B(0,R)\cap\Sigma=\varnothing.
\]
By local convergence, for all sufficiently large
$j$ each $\sigma\in E$ has a unique corresponding
node $\sigma_j\in\Sigma_{k_j}$ such that
$\sigma_j\to\sigma$. Write
\[
  \lambda_{\sigma,j}
  =z^{-1}(\sigma_j/\sqrt{k_j})
  \in\mathcal F_{k_j}.
\]

At $\lambda_{\sigma,j}$, prescribe the value
$d_\sigma$ relative to the local unit frame
\[
  u_{k_j}(z)
  =e(z)^{\otimes k_j}e^{k_j\varphi(z)/2}.
\]
Prescribe zero at every other Fekete point,
using arbitrary unit frames outside the coordinate
neighborhood. Let $s_j\in H_{k_j}$ be the unique
interpolating section. Then
\[
  \|s_j\|_\infty\le M.
\]

Write $s_j=f_j e^{\otimes k_j}$ locally and define
\[
  F_j(w)=f_j(w/\sqrt{k_j}).
\]
These functions are holomorphic on expanding
neighborhoods of $\mathbb C^n$ and satisfy
\[
  |F_j(w)|
  e^{-k_j\varphi(w/\sqrt{k_j})/2}\le M.
\]
Since
\[
  k_j\varphi(w/\sqrt{k_j})
  \longrightarrow |w|^2
\]
uniformly on compact subsets, the functions
$F_j$ form a normal family. By Montel's theorem
and a diagonal subsequence argument, they
converge locally uniformly to an entire function
$F$ satisfying
\[
  \|F\|_{\mathcal B^\infty}
  =\sup_{w\in\mathbb C^n}|F(w)|e^{-|w|^2/2}
  \le M.
\]

For $\sigma\in E$, the prescribed interpolation
condition is
\[
  F_j(\sigma_j)e^{-k_j\varphi(
       \sigma_j/\sqrt{k_j})/2}
  =d_\sigma.
\]
Passing to the limit gives
\[
  F(\sigma)e^{-|\sigma|^2/2}=d_\sigma.
\]
On the other hand, if $\sigma\in\Sigma\setminus E$,
local convergence supplies a sequence of matching
nodes $\sigma_j\in\Sigma_{k_j}$ with
$\sigma_j\to\sigma$. For all sufficiently large $j$
these nodes are distinct from those corresponding
to $E$, and hence $F_j(\sigma_j)=0$ by construction.
Consequently,
\[
  F(\sigma)=0
  \qquad(\sigma\in\Sigma\setminus E).
\]
Thus every finitely supported family of normalized
data on $\Sigma$ has an interpolant in
$\mathcal B^\infty$ with norm at most $M$.

Finally, let $(d_\sigma)_{\sigma\in\Sigma}$ be
arbitrary bounded normalized data. Choose an
increasing sequence of finite subsets
$E_m\uparrow\Sigma$, and apply the preceding
argument to the data restricted to $E_m$, extended
by zero outside $E_m$. The resulting entire
functions have $\mathcal B^\infty$ norms bounded
by $M$. A second normal-family argument yields
an entire interpolant of all the prescribed data,
with the same norm bound. Scaling handles data
of arbitrary finite supremum norm.

Therefore $\Sigma$ is an interpolating sequence
for $\mathcal B^\infty$. By the strict density
theorem (3.4),
\[
  D^+(\Sigma)<1,
\]
contradicting Lemma 3.2. Hence no bounded
subsequence of $\Lambda_k$ exists, and
\[
  \Lambda_k\longrightarrow+\infty.
\]

For the final assertion of Theorem 1.1, the case
$a>1$ follows from part (i), while $a=1$ fails
by part (ii) and (3.2). If $0<a<1$, then
$N_{\lceil ak\rceil}<N_k$ for all sufficiently
large $k$, by (2.1). The evaluation map from
$H_k$ to the fibers over $\mathcal F_{\lceil ak\rceil}$
therefore has a nonzero kernel, so
$\Lambda_k(a)=+\infty$.
\end{proof}

Next, we establish a quantitative lower bound for the Lebesgue constants $\Lambda_k$ on the Riemann sphere $\Bbb{P}^1$:

\begin{proposition}[A lower bound on $\mathbb P^1$]
\label{prop:P1-rate}
Let $X=\mathbb P^1$ and $L=\mathcal O(1)$ with the
Fubini--Study metric. Then the Lebesgue constants of
Fekete configurations satisfy
\begin{equation}
  \Lambda_k \ge
  \bigl(e^{\frac12-\log 2+\Delta}+o(1)\bigr)\sqrt{N_k}
  \ge (0.65+o(1))\sqrt{N_k},
\end{equation}
where $N_k=k+1$ and $\Delta=-0.2255\ldots$.
The same lower bound holds for the average of the
Lebesgue function over $\mathbb P^1$.
\end{proposition}

\begin{proof}
Let $X=\mathbb P^1$, $L=\mathcal O(1)$ with the Fubini--Study metric, and identify
$\mathbb P^1$ with the unit sphere $S^2\subset\mathbb R^3$, with normalized surface
measure $\sigma$ and chordal distance $|x-y|$. Then $N_k=k+1$, and
$|\operatorname{vdm}_k(x_1,\dots,x_{N_k})|_{h^k}$ is a constant multiple of
$\prod_{i<j}|x_i-x_j|$. Hence Fekete configurations are exactly the minimizers of the
logarithmic energy
$$E_{\log}(x_1,\dots,x_N)=\sum_{i\ne j}\log|x_i-x_j|^{-1}$$ and
$|\ell^k_j(x)|=\prod_{i\ne j}|x-x_i|/|x_j-x_i|$.
Set $f(x)=N^{-1}\sum_j\log|\ell^k_j(x)|$, which is $\sigma$-integrable. By the
arithmetic--geometric mean inequality and Jensen's inequality we have
\[
  \Lambda_k\ \ge\ \int_{S^2}\sum_{j=1}^{N}|\ell^k_j|\,d\sigma
  \ \ge\ N\exp\Bigl(\int_{S^2} f\,d\sigma\Bigr),
  \qquad N=N_k .
\]
Since $\sum_j\log|\ell^k_j(x)|=(N-1)\sum_i\log|x-x_i|+E_{\log}(\mathcal F_k)$ and
$\int_{S^2}\log|x-y|\,d\sigma(x)=\log 2-\tfrac12$, we obtain
$\int f\,d\sigma=(N-1)(\log 2-\tfrac12)+N^{-1}E_{\log}(\mathcal F_k)$.
By the lower bound of Rakhmanov, Saff and Zhou \cite{RSZ94} (see also
\cite[Section~3]{BHS12}),
\begin{equation}
  E_{\log}(\mathcal F_k)\ \ge\ \bigl(\tfrac12-\log 2\bigr)N^2-\tfrac12N\log N
  +(\Delta+o(1))\,N,\qquad \Delta=-0.2255\ldots .
\end{equation}
Combining these estimates gives
\[
  \Lambda_k\ \ge\ \bigl(e^{\frac12-\log 2+\Delta}+o(1)\bigr)\sqrt{N_k}
  \ \ge\ (0.65+o(1))\sqrt{N_k}.
\]
The same bound holds for the average $\int_{S^2}\Lambda_k(x)\,d\sigma(x)$ of the
Lebesgue function.
\end{proof}

It would be interesting to know a matching upper bound, or an analogue of the estimate in Proposition~\ref{prop:P1-rate} in higher dimensions.
\subsection{Extremal sections}
The standard phase-alignment construction for the norm of an interpolation
operator makes the critical-level obstruction explicit.

\begin{proposition}[Extremal sequences]\label{prop:extremal}
For every sufficiently large $k$, there is a section $s_k\in H_k$ such that
\[
 |s_k(x_j)|=1\quad(x_j\in\F_k),
 \qquad \norm{s_k}_\infty=\Lambda_k.
\]
Consequently, $t_k=s_k/\Lambda_k$ satisfies
\[
 \norm{t_k}_\infty=1,
 \qquad \max_{x\in\F_k}|t_k(x)|=\Lambda_k^{-1}\longrightarrow0.
\]
\end{proposition}
\begin{proof}
Choose $y_k\in X$ maximizing the Lebesgue function
$\sum_{j=1}^{N_k}|\ell_j^k(y)|$, and choose a unit vector
$v_k\in L^k_{y_k}$. Write $\ell_j^k(y_k)=z_{k,j}v_k$ and set
\[
 a_{k,j}=
 \begin{cases}
 \overline{z_{k,j}}/|z_{k,j}|,&z_{k,j}\ne0,\\
 1,&z_{k,j}=0,
 \end{cases}
 \qquad
 s_k=\sum_{j=1}^{N_k}a_{k,j}\ell_j^k.
\]
The Lagrange identities give $s_k(x_j)=a_{k,j}u_j$, hence
$|s_k(x_j)|=1$. At $y_k$ all nonzero summands have the same phase, so
$|s_k(y_k)|=\sum_j|z_{k,j}|=\Lambda_k$.
For every $x\in X$, the triangle inequality gives
$|s_k(x)|\leq\sum_j|\ell_j^k(x)|\leq\Lambda_k$.
The final assertion follows from Theorem~\ref{prop:deterministic}(ii).
\end{proof}

\begin{remark}\label{rem:extremal-phases}
Conversely, suppose that $\max_{\F_k}|s|=1$ and
$\norm{s}_\infty=\Lambda_k$, and choose $y$ with $|s(y)|=\Lambda_k$.
Writing $s(x_j)=a_ju_j$, equality must hold throughout
\[
 \Lambda_k
 =\left|\sum_j a_j\ell_j^k(y)\right|
 \leq\sum_j|a_j|\,|\ell_j^k(y)|
 \leq\sum_j|\ell_j^k(y)|
 \leq\Lambda_k.
\]
Thus $y$ maximizes the Lebesgue function, and whenever
$\ell_j^k(y)\ne0$ one has $|a_j|=1$ and all vectors
$a_j\ell_j^k(y)$ have the same phase. The construction therefore
describes the alignment responsible for exact extremality.
It depends on the chosen Fekete configurations and supplies no
quantitative rate for $\Lambda_k\to\infty$.
\end{remark}

\section{Random holomorphic sections}\label{sec:random}
For a fixed $L^2$ orthonormal basis $\{S_j^k\}$ of $H^0(X,L^k)$ we write
\begin{equation}\label{eq:gaussian}
s=\sum_{j=1}^{N_k}c_jS_j^k,\qquad
c_j=N_k^{-1/2}g_j,
\end{equation}
where $g_j$ are independent standard complex Gaussians, with density $\pi^{-1}e^{-|z|^2}$ on $\C$. This induces a probability measure on $H^0(X,L^k)$ via
\[
d\Prob_k(c)=(N_k/\pi)^{N_k}e^{-N_k\norm{c}^2}\,d\mathcal L(c),\quad
\E(c_i\overline{c_j})=N_k^{-1}\delta_{ij},\quad \E(c_ic_j)=0.
\]
In particular, $\E\norm{s}_2^2=1$. Unitary invariance makes this measure independent of the choice of orthonormal basis.

The radial projection $s\mapsto s/\norm{s}_2$ pushes $\Prob_k$ forward to normalized surface probability measure on the $L^2$ unit sphere to which we refer as spherical ensemble. Consequently, any measurable event invariant under positive radial scaling has the same probability in the two ensembles. This applies to the norm-comparison failure events in Theorem~\ref{thm:random} and Theorem~\ref{thm:L2}. It does not assert equality of Gaussian probabilities and probabilities of radial images for arbitrary measurable sets.

For background on random holomorphic sections and their zeros, see \cite{SZ,B6,BCHM} and the references therein. Supremum estimates for such fields were obtained in \cite{SZ1,FZ}. We first establish general upper and lower bounds, then treat uniform variance in Section~\ref{subsec:sharp}, and finally use equidistribution to obtain exact Fekete norming for arbitrary background forms in Section~\ref{subsec:fekete-exact}.

\subsection{An upper tail for the supremum}
\begin{lemma}\label{lem:sup-tail}
There are geometric constants $C,b_+>0$ such that, for all sufficiently large $k$ and $t\geq0$,
\[
\Prob_k\{\norm{s}_\infty>t\}\leq Ck^n e^{-t^2/(4b_+)}.
\]
In particular, for every $q>0$ there is $A_q>0$ such that
\[
\Prob_k\{\norm{s}_\infty>A_q\sqrt{\log k}\}=O(k^{-q}).
\]
\end{lemma}
\begin{proof}
At any point $x$, express $s(x)$ in a unit vector of the fiber. The resulting scalar is a centered complex Gaussian of variance $B_k(x)/N_k$. Equations~\eqref{eq:dimension} and \eqref{eq:diagonal} give constants $0<b_-\leq b_+$ such that
\begin{equation}\label{eq:variance-bounds}
b_-\leq B_k(x)/N_k\leq b_+
\end{equation}
uniformly in $x$ and large $k$.
Choose a net $T_k$ of mesh $\eta/\sqrt k$ and cardinality at most $C_\eta k^n$, where $\eta>0$ is fixed sufficiently small. By \eqref{eq:bernstein},
\[
\norm{s}_\infty\leq2\max_{x\in T_k}|s(x)|.
\]
Since a centered complex Gaussian of variance $v$ satisfies $\Prob\{|Y|>u\}=e^{-u^2/v}$, a union bound proves the first assertion. Choose $A_q=2\sqrt{b_+(n+q)}$ for the second one.
\end{proof}

\subsection{A lower tail for the sampled maximum}
Fix a deterministic array $\Gamma_k$ satisfying \eqref{eq:general-array}. Constants in the remainder of this section may depend on $\delta_0,c_0$ as well as the fixed geometry. After reducing $\delta_0$ we may assume $\delta_0\leq1$.

\begin{lemma}\label{lem:covariance}
There exist $D,v,m_0>0$ and subsets $M_k\subset\Gamma_k$ such that
\[
m_k:=\#M_k\geq m_0\frac{k^n}{(\log k)^{2n}},\qquad
 d(\lambda,\mu)\geq D\frac{\log k}{\sqrt k}\quad(\lambda\ne\mu\in M_k),
\]
and the covariance matrix $\Sigma_k$ of the unit-frame evaluations $(s(\lambda))_{\lambda\in M_k}$ satisfies $\Sigma_k\geq v I$ for all sufficiently large $k$.
\end{lemma}
\begin{proof}
Take $M_k$ maximal with the indicated separation. Its balls of radius $D\log k/\sqrt k$ cover $\Gamma_k$. Lemma~\ref{lem:packing} bounds the number of points of $\Gamma_k$ in each such ball by $C_D(\log k)^{2n}$, proving the cardinality estimate.

Choose unit vectors in the fibers over $M_k$ and let $Y_\lambda$ be the scalar evaluations in these frames. The vector $Y$ is a centered proper complex Gaussian vector, since \eqref{eq:gaussian} gives $\E(Y_\lambda Y_\mu)=0$. Its covariance obeys
\[
(\Sigma_k)_{\lambda\lambda}=B_k(\lambda)/N_k\geq b_-,\qquad
|(\Sigma_k)_{\lambda\mu}|=|K_k(\lambda,\mu)|/N_k
\leq C e^{-c\sqrt k\,d(\lambda,\mu)}.
\]
For distinct points of $M_k$ the last quantity is at most $Ck^{-cD}$. Since $m_k\leq Ck^n$, every off-diagonal row sum is at most $Ck^{n-cD}$. Choose $D$ with $cD>n+1$. For large $k$, the row sums are at most $b_-/2$, so $\Sigma_k\geq(b_-/2)I$ by the Hermitian diagonal-dominance bound. Putting $v=b_-/2$ finishes the proof.
\end{proof}

\begin{lemma}\label{lem:lower-tail}
There are constants $\tau,c>0$ such that
\begin{equation}\label{eq:lower-tail}
\Prob_k\left\{\max_{\lambda\in\Gamma_k}|s(\lambda)|
\leq\tau\sqrt{\log k}\right\}
\leq\exp\left(-c\frac{k^{n/2}}{(\log k)^{2n}}\right)
\end{equation}
for all sufficiently large $k$.
\end{lemma}
\begin{proof}
Let $M_k\subset \Gamma_k$ and $m_k=\# M_k$ be as in  Lemma~\ref{lem:covariance}. Then the vector of evaluations on $M_k$ has the distribution
\[
Y=\sqrt v\,Z+W,
\]
where $Z$ has independent standard complex Gaussian coordinates and $W$ is an independent proper complex Gaussian vector of covariance $\Sigma_k-vI$. Such a decomposition exists because this matrix is positive semidefinite by Lemma~\ref{lem:covariance}.

For any $w\in\C$, the integral of the radially decreasing Gaussian density over a disk of fixed radius is largest when the disk is centered at the origin. Therefore
\[
\Prob\{|\sqrt v Z_j+w|\leq t\}\leq1-e^{-t^2/v}.
\]
Conditioning on $W$ and using independence of the coordinates of $Z$ yields
\[
\Prob\{\max_{\lambda\in M_k}|Y_\lambda|\leq t\}
\leq(1-e^{-t^2/v})^{m_k}\leq\exp(-m_ke^{-t^2/v}).
\]
Take $\tau=\sqrt{vn/2}$ and $t=\tau\sqrt{\log k}$. Then $e^{-t^2/v}=k^{-n/2}$, and the cardinality bound proves \eqref{eq:lower-tail}.
\end{proof}

\begin{theorem}\label{thm:general-random}
For any array satisfying \eqref{eq:general-array} and every $q>0$, there is $C_q>0$ such that
\[
\Prob_k\left\{\norm{s}_\infty>C_q\max_{\lambda\in\Gamma_k}|s(\lambda)|\right\}
=O(k^{-q}).
\]
\end{theorem}
\begin{proof}
Take $A_q=2\sqrt{b_+(n+q)}$ as in Lemma~\ref{lem:sup-tail}, and set
\[
C_q=\frac{A_q}{\tau}=\sqrt{\frac{8b_+(n+q)}{vn}}.
\]
Thus $C_q=O(\sqrt{1+q})$, with the geometry and array fixed. The $q$ dependence comes from imposing a polynomial failure bound on the upper tail of $\norm{s}_\infty$; the sampled maximum has a superpolynomially small lower tail with a threshold independent of $q$. If $\norm{s}_\infty\leq A_q\sqrt{\log k}$ and the sampled maximum is at least $\tau\sqrt{\log k}$, the desired norm comparison holds. The union of the two complementary events has probability $O(k^{-q})$ by Lemmas~\ref{lem:sup-tail} and \ref{lem:lower-tail}; the latter bound is smaller than every negative power of $k$.
\end{proof}

\begin{proof}[Proof of Theorem~\ref{thm:random}]
Fekete arrays satisfy \eqref{eq:general-array} by \eqref{eq:dimension} and Proposition~\ref{prop:sep}. Applying Theorem~\ref{thm:general-random} we obtain (\ref{eq:random-main}). For $q>1$, the probabilities of failure are summable. The Borel--Cantelli lemma gives the assertion for all sufficiently large $k$, almost surely. 
\end{proof}

\begin{corollary}[Rarity of near-extremizers]\label{cor:near-extremizers}
Fix $\theta\in(0,1]$. In either the Gaussian or spherical ensemble,
for every $q>0$,
\[
 \Prob_k\left\{
 \frac{\norm{s}_\infty}{\max_{x\in\F_k}|s(x)|}
 \geq\theta\Lambda_k\right\}=O(k^{-q}).
\]
The implicit constant may depend on $\theta$ and $q$.
\end{corollary}
\begin{proof}
The denominator is positive for every nonzero $s\in H_k$, since
$\F_k$ is unisolvent. Fix $q>0$ and let $C_q$ be the constant in
Theorem~\ref{thm:random}. Since $\Lambda_k\to\infty$, one has
$\theta\Lambda_k>C_q$ for all sufficiently large $k$.
The displayed event is then contained in the norm-comparison failure
event of that theorem. The event is invariant under positive radial
scaling, so the same bound holds in the spherical ensemble.
\end{proof}

In particular, the sections constructed in
Proposition~\ref{prop:extremal} attain the deterministic obstruction,
whereas the probability of attaining any fixed positive fraction of
the extremal ratio decays faster than every fixed inverse power of $k$.

\begin{remark}\label{rem:undersampling}
For every fixed $a>0$, the array $\Gamma_k=\F_{\lceil ak\rceil}$ satisfies \eqref{eq:general-array}, with constants depending on $a$. Consequently Theorem~\ref{thm:general-random} also holds on this array, including when $0<a<1$. In particular, it applies to $\F_{\lfloor(1-\epsilon)k\rfloor}$ for every fixed $0<\epsilon<1$. For $a<1$, deterministic norming is impossible by dimension counting, but the random estimate is consistent with that obstruction: it controls typical sections and allows an exceptional set. No uniform assertion as $a\downarrow0$ is made.
\end{remark}

\subsection{Asymptotically exact norming under a uniform variance hypothesis}
\label{subsec:sharp}

In this subsection we assume
\begin{equation}\label{eq:H}
  c_1(L,h)^n=\kappa\,\omega^n
  \qquad\text{for a constant }\kappa>0.
\end{equation}
Equivalently, $b_0\equiv\kappa$ in \eqref{eq:diagonal}.
Writing $V_\omega=\int_XdV$, one has $V_L=\kappa V_\omega$.
Thus the limiting pointwise variance is
$\kappa/V_L=V_\omega^{-1}$. Put
\begin{equation}
  \sigma_L^2 := \frac{\kappa}{V_L}=\frac1{V_\omega}, \qquad
  \sigma_k^2(x) := \frac{B_k(x)}{N_k} = \mathbb E\,|s(x)|_{h^k}^2 .
\end{equation}
By \eqref{eq:dimension} and \eqref{eq:diagonal} there is $C_0>0$ such that, for all sufficiently large $k$,
\begin{equation}\label{eq:var}
  \bigl|\sigma_k^2(x)-\sigma_L^2\bigr| \le C_0\,\sigma_L^2\,k^{-1}
  \qquad (x\in X).
\end{equation}
Throughout we write
\[
  \Theta_k := \sigma_L\sqrt{n\log k}.
\]

\begin{lemma}[Refined upper-tail estimate]\label{lem:sharp-upper}
Assume \eqref{eq:H}. For every $1<\alpha'<\alpha$ there is $C=C(\alpha,\alpha')$
such that, for all sufficiently large $k$,
\[
  \mathbb P_k\bigl\{\|s\|_\infty>\alpha\,\Theta_k\bigr\}
  \le C\,k^{-n(\alpha'^2-1)}.
\]
\end{lemma}

\begin{proof}
Let $C_1$ be the constant in \eqref{eq:bernstein}, and fix $\eta\in(0,1]$ so small that
$(1-C_1\eta)\alpha\ge\alpha'$. Let $T_k\subset X$ be a maximal
$\eta k^{-1/2}$-separated set. By maximality, every $x\in X$ satisfies
$d(x,T_k)<\eta k^{-1/2}$, and by \eqref{eq:packing} of Lemma~\ref{lem:packing}, applied with $m=k$,
$\rho=\eta$ and $t=\operatorname{diam}X$,
\[
  \#T_k\le C\eta^{-2n}k^n .
\]
For $x\in X$ choose $y\in T_k$ with $d(x,y)<\eta k^{-1/2}$. Then \eqref{eq:bernstein} gives
$|s(x)|\le|s(y)|+C_1\eta\|s\|_\infty$, hence
\[
  \|s\|_\infty \le (1-C_1\eta)^{-1}\max_{y\in T_k}|s(y)| .
\]
Consequently
\[
  \{\|s\|_\infty>\alpha\Theta_k\}\subset
  \Bigl\{\max_{y\in T_k}|s(y)|>\alpha'\Theta_k\Bigr\}.
\]
For fixed $y$, the unit-frame evaluation of $s(y)$ is a centered complex
Gaussian of variance $\sigma_k^2(y)$, so
$\mathbb P\{|s(y)|>u\}=\exp(-u^2/\sigma_k^2(y))$. By \eqref{eq:var},
\[
  \frac{\alpha'^2\Theta_k^2}{\sigma_k^2(y)}
  \ge \frac{\alpha'^2 n\log k}{1+C_0k^{-1}}
  \ge \alpha'^2 n\log k-\alpha'^2 nC_0\,\frac{\log k}{k}
  \ge \alpha'^2 n\log k - C.
\]
A union bound over $T_k$ yields
\[
  \mathbb P_k\{\|s\|_\infty>\alpha\Theta_k\}
  \le C\eta^{-2n}k^n\,e^{C}k^{-n\alpha'^2}
  = C(\alpha,\alpha')\,k^{-n(\alpha'^2-1)} .
\]
\end{proof}

\begin{lemma}[Lower tail of the sampled maximum]\label{lem:sharp-lower}
Assume \eqref{eq:H}, and let $\Gamma_k$ satisfy \eqref{eq:general-array}. For every $\beta\in(0,1)$
there is $c=c(\beta)>0$ such that, for all sufficiently large $k$,
\[
  \mathbb P_k\Bigl\{\max_{\lambda\in\Gamma_k}|s(\lambda)|<\beta\,\Theta_k\Bigr\}
  \le \exp\Bigl(-c\,\frac{k^{n(1-\beta^2)}}{(\log k)^{2n}}\Bigr).
\]
\end{lemma}

\begin{proof}
Let $M_k\subset\Gamma_k$ and $D$ be as in Lemma~\ref{lem:covariance}, so that
$m_k=\#M_k\ge m_0k^n(\log k)^{-2n}$ and $cD>n+1$. The proof of Lemma~\ref{lem:covariance}
shows that every off-diagonal row sum of the covariance matrix $\Sigma_k$ of
the unit-frame evaluations on $M_k$ is at most $Ck^{n-cD}\le Ck^{-1}$, while
by \eqref{eq:var} its diagonal entries are at least $\sigma_L^2(1-C_0k^{-1})$.
By the Hermitian diagonal-dominance bound,
\[
  \Sigma_k\ge v_k I,\qquad v_k:=\sigma_L^2\bigl(1-C_2k^{-1}\bigr),
\]
for some $C_2>0$ and all large $k$. The argument of Lemma~\ref{lem:lower-tail}, with $v$
replaced by $v_k$, gives for every $t>0$
\[
  \mathbb P\Bigl\{\max_{\lambda\in M_k}|s(\lambda)|\le t\Bigr\}
  \le \exp\bigl(-m_k\,e^{-t^2/v_k}\bigr).
\]
Take $t=\beta\Theta_k$. Then
\[
  \frac{t^2}{v_k}=\frac{\beta^2 n\log k}{1-C_2k^{-1}}
  \le \beta^2 n\log k + C,
\]
so $m_k e^{-t^2/v_k}\ge c\,k^{n(1-\beta^2)}(\log k)^{-2n}$. Since
$M_k\subset\Gamma_k$, the claim follows.
\end{proof}

\begin{theorem}\label{thm:sharp}
Assume \eqref{eq:H}, and let $\Gamma_k$ satisfy \eqref{eq:general-array}. Put
\[
  R_k(s):=\frac{\|s\|_\infty}{\max_{\lambda\in\Gamma_k}|s(\lambda)|}\in[1,\infty],
\]
The quotient is used when the denominator is positive; set $R_k(0)=1$ and $R_k(s)=+\infty$ for nonzero sections whose sampled values all vanish.
\begin{enumerate}
\item[(i)] For every $\varepsilon\in(0,1)$ there is $C_\varepsilon$ such that,
for all sufficiently large $k$,
\[
  \mathbb P_k\Bigl\{(1-\varepsilon)\Theta_k\le \max_{\Gamma_k}|s|
  \le \|s\|_\infty\le(1+\varepsilon)\Theta_k\Bigr\}
  \ge 1-C_\varepsilon k^{-n\varepsilon}.
\]
In particular $R_k\to1$ in probability, and
$\|s\|_\infty/\Theta_k\to1$ and $\max_{\Gamma_k}|s|/\Theta_k\to1$ in probability.
\item[(ii)] For every $q>0$ and every $C>\sqrt{1+q/n}$,
\[
  \mathbb P_k\{R_k(s)>C\}=O(k^{-q}).
\]
\item[(iii)] For every $C>\sqrt{1+1/n}$, almost every sequence $(s_k)$ under the
product measure satisfies $R_k(s_k)\le C$ for all sufficiently large $k$. Thus
$\limsup_{k\to\infty}R_k(s_k)\le\sqrt{1+1/n}$ almost surely.
\end{enumerate}
The conclusions concerning $R_k$ also hold in the spherical ensemble, as do the two normalized-amplitude convergence statements in part~(i).
\end{theorem}

\begin{proof}
The inequality $\max_{\Gamma_k}|s|\le\|s\|_\infty$ is trivial, so
$R_k\ge1$.

(i) Apply Lemma~\ref{lem:sharp-upper} with $\alpha=1+\varepsilon$ and
$\alpha'=\sqrt{1+\varepsilon}\in(1,\alpha)$; the exponent is
$n(\alpha'^2-1)=n\varepsilon$. Apply Lemma~\ref{lem:sharp-lower} with
$\beta=1-\varepsilon$; since $1-\beta^2=2\varepsilon-\varepsilon^2\ge\varepsilon$,
the resulting bound $\exp(-ck^{n\varepsilon}(\log k)^{-2n})$ is
$O(k^{-n\varepsilon})$. The union of the two exceptional events therefore has
probability at most $C_\varepsilon k^{-n\varepsilon}$. On the complement,
$1\le R_k\le(1+\varepsilon)/(1-\varepsilon)$. Since $\varepsilon$ is arbitrary,
the three convergence statements follow.

(ii) Since $C>\sqrt{1+q/n}$, choose $\beta\in(0,1)$ with
$C\beta>\sqrt{1+q/n}$, and then $\alpha'$ with
$\sqrt{1+q/n}<\alpha'<C\beta=:\alpha$. Then $n(\alpha'^2-1)>q$. If
$\|s\|_\infty\le\alpha\Theta_k$ and $\max_{\Gamma_k}|s|\ge\beta\Theta_k$, then
$R_k(s)\le\alpha/\beta=C$. Hence, by Lemmas~\ref{lem:sharp-upper} and
\ref{lem:sharp-lower},
\[
  \mathbb P_k\{R_k>C\}
  \le C(\alpha,\alpha')\,k^{-n(\alpha'^2-1)}
  +\exp\Bigl(-c\,\frac{k^{n(1-\beta^2)}}{(\log k)^{2n}}\Bigr)
  =O(k^{-q}).
\]

(iii) Given $C>\sqrt{1+1/n}$, choose $q>1$ with $C>\sqrt{1+q/n}$. By (ii),
$\sum_k\mathbb P_k\{R_k>C\}<\infty$, and the Borel--Cantelli lemma applies;
independence between levels is not needed. Taking $C\downarrow\sqrt{1+1/n}$
along a countable sequence gives the $\limsup$ bound.

To justify the ensemble assertions, let $g_k$ have law $\Prob_k$ and put $u_k=g_k/\norm{g_k}_2$, which has the spherical distribution. Since $R_k(u_k)=R_k(g_k)$ almost surely, the ratio probabilities agree exactly, and the almost-sure ratio bound follows again by Borel--Cantelli. Moreover,
\[
\E\norm{g_k}_2^2=1,\qquad
\operatorname{Var}(\norm{g_k}_2^2)=N_k^{-1},
\]
so $\norm{g_k}_2\to1$ in probability. Dividing the Gaussian amplitude limits in part~(i) by $\norm{g_k}_2$ gives the corresponding spherical limits. These amplitude limits use norm concentration in addition to radial invariance.
\end{proof}

\begin{remark}\label{rem:sharp}
\begin{enumerate}
\item[(a)] Theorem~\ref{thm:sharp}(ii) refines Theorem~\ref{thm:general-random} under
\eqref{eq:H}: the norming constant may be taken arbitrarily close to
$\sqrt{1+q/n}$, which is consistent with, and sharpens,
$C_q=O(\sqrt{1+q})$. In particular, for every $C>1$,
$\mathbb P_k\{R_k>C\}\to0$, whereas the deterministic constant
$\Lambda_k=\sup_{s\in H_k\setminus\{0\}}R_k(s)$ for $\Gamma_k=\mathcal F_k$ tends to $+\infty$ by
Theorem~\ref{prop:deterministic}(ii). We do not claim that the threshold $\sqrt{1+q/n}$ is optimal. Establishing optimality would require lower bounds for the tail of the ratio $R_k$, taking account of the dependence between the supremum and the sampled maximum.
\item[(b)] The leading order $\|s\|_\infty\sim\sigma_L\sqrt{n\log k}$ in
probability is consistent with the asymptotics of the median and mean of the
supremum in \cite{FZ}.
\item[(c)] Without \eqref{eq:H}, the limiting variance $b_0/V_L$
need not be constant. Section~\ref{subsec:fekete-exact} proves that Fekete
equidistribution still yields asymptotically exact norming, and constructs
separated arrays for which exactness fails.
\end{enumerate}
\end{remark}

\subsection{Fekete equidistribution and nonconstant variance}
\label{subsec:fekete-exact}
We now allow an arbitrary background K\"ahler form $\omega$. Recall the
notation $\sigma_{\max}$ and $\Theta_k^{\max}$ from the introduction. The
diagonal expansion gives, uniformly on $X$,
\begin{equation}\label{eq:variable-variance}
 \frac{B_k(x)}{N_k}=\frac{b_0(x)}{V_L}+O(k^{-1}),
 \qquad
 \frac{B_k(x)}{N_k}\leq\sigma_{\max}^2(1+Ck^{-1}).
\end{equation}
The proof of Lemma~\ref{lem:sharp-upper} uses only this upper variance
bound. Replacing $\sigma_L$ by $\sigma_{\max}$ there therefore yields
\begin{equation}\label{eq:variable-upper}
 \Prob_k\{\norm{s}_\infty>\alpha\Theta_k^{\max}\}
 \leq C_{\alpha,\alpha'}k^{-n(\alpha'^2-1)}
 \qquad(1<\alpha'<\alpha).
\end{equation}

We isolate the localized lower estimate used both for Fekete arrays and
for the counterexample below.
\begin{lemma}\label{lem:localized-maximum}
Let $\Delta_k$ be $\delta_0k^{-1/2}$-separated with
$\#\Delta_k\geq c_0k^n$. Suppose that, for some fixed $w>0$,
\begin{equation}
 B_k(\lambda)/N_k\geq w-Ck^{-1}
 \qquad(\lambda\in\Delta_k).
\end{equation}
For every $\beta\in(0,1)$ there is $c>0$ such that
\[
 \Prob_k\left\{\max_{\Delta_k}|s|
       <\beta\sqrt{wn\log k}\right\}
 \leq\exp\left(-c\frac{k^{n(1-\beta^2)}}{(\log k)^{2n}}\right)
\]
for all sufficiently large $k$.
\end{lemma}
\begin{proof}
As in Lemma~\ref{lem:covariance}, take a maximal
$D\log k/\sqrt{k}$-separated subset $M_k\subset\Delta_k$, with
$cD>n+1$ for the constant in \eqref{eq:off-diagonal}.
Packing gives $\#M_k\geq c k^n/(\log k)^{2n}$, and the
off-diagonal covariance row sums are $O(k^{-1})$. Hence the covariance
matrix is bounded below by $v_kI$, where $v_k=w-C'k^{-1}>0$.
The conditional Gaussian argument of Lemma~\ref{lem:lower-tail} gives
\[
 \Prob_k\{\max_{M_k}|s|\leq t\}
 \leq\exp\bigl(-\#M_k\,e^{-t^2/v_k}\bigr).
\]
For $t=\beta\sqrt{wn\log k}$,
$t^2/v_k\leq\beta^2n\log k+C''$, proving the assertion.
\end{proof}

\begin{theorem}[Exact Fekete norming for arbitrary background forms]
\label{thm:fekete-exact}
Fix $a>0$ and put $\Gamma_k=\F_{\lceil ak\rceil}$. 
\begin{enumerate}
\item[(i)] For every $\varepsilon\in(0,1)$ and all sufficiently large $k$,
\[
 \Prob_k\left\{(1-\varepsilon)\Theta_k^{\max}
 \leq\max_{\Gamma_k}|s|\leq\norm{s}_\infty
 \leq(1+\varepsilon)\Theta_k^{\max}\right\}
 \geq1-C_\varepsilon k^{-n\varepsilon}.
\]
In particular the two amplitudes divided by $\Theta_k^{\max}$, and their
ratio $R_k$, converge to one in probability.
\item[(ii)] For every $\beta\in(0,1)$ and
$0<\gamma<n(1-\beta^2)$, there is $c_{\beta,\gamma}>0$ such that
\[
 \Prob_k\{\max_{\Gamma_k}|s|<\beta\Theta_k^{\max}\}
 \leq e^{-c_{\beta,\gamma}k^\gamma}
\]
for all sufficiently large $k$.
\item[(iii)] For every $q>0$ and $C>\sqrt{1+q/n}$,
$\Prob_k\{R_k>C\}=O(k^{-q})$. Consequently,
\[
 \limsup_{k\to\infty}R_k(s_k)\leq\sqrt{1+1/n}
 \quad\text{almost surely under the product measure}.
\]
\end{enumerate}
Constants may depend on $a$ and on the fixed geometry.
The ratio is interpreted as in Theorem~\ref{thm:exact-intro}.
The ratio assertions and the amplitude convergence statements also hold
in the spherical ensemble.
\end{theorem}
\begin{proof}
Write $m=\lceil ak\rceil$. The equidistribution theorem
\cite[Theorem A]{BBWN}, specialized to a smooth positive metric on $X$,
gives
\begin{equation}\label{eq:fekete-weak}
 \frac1{N_m}\sum_{\lambda\in\F_m}\delta_\lambda
 \ \longrightarrow\
 \mu_{\rm eq}:=\frac{c_1(L,h)^n}{n!V_L}
 \quad\text{weakly}.
\end{equation}
The measure $\mu_{\rm eq}$ has a smooth strictly positive density.
For each fixed $\eta\in(0,1)$ the set
\[
 U_\eta=\{x:b_0(x)>(1-\eta)\max_Xb_0\}
\]
is nonempty and open. The open-set inequality for weak convergence and
$N_m\sim V_La^nk^n$ show that
\begin{equation}\label{eq:high-variance-count}
 \#(\F_m\cap U_\eta)\geq c_{a,\eta}k^n
\end{equation}
for all sufficiently large $k$.
Proposition~\ref{prop:sep} gives separation at scale $k^{-1/2}$, with
a constant depending on $a$. Moreover, on these nodes
\[
 B_k(\lambda)/N_k\geq(1-\eta)\sigma_{\max}^2-O(k^{-1}).
\]
For $\beta^2<1-\eta$, apply Lemma~\ref{lem:localized-maximum} with
$w=(1-\eta)\sigma_{\max}^2$ and parameter
$\beta/\sqrt{1-\eta}$. Since the selected nodes are contained in $\Gamma_k$,
\begin{equation}\label{eq:fekete-lower-eta}
 \Prob_k\{\max_{\Gamma_k}|s|<\beta\Theta_k^{\max}\}
 \leq
 \exp\left(-c_{a,\eta,\beta}
 \frac{k^{\,n(1-\beta^2/(1-\eta))}}{(\log k)^{2n}}\right).
\end{equation}
Given $\gamma<n(1-\beta^2)$, choose a fixed $\eta>0$ so small that
$n(1-\beta^2/(1-\eta))>\gamma$. Absorbing the logarithmic factor
proves (ii). This argument uses only fixed open sets in
\eqref{eq:fekete-weak}, not a rate of equidistribution.

For (i), apply \eqref{eq:variable-upper} with
$\alpha=1+\varepsilon$ and $\alpha'=\sqrt{1+\varepsilon}$.
Apply (ii) with $\beta=1-\varepsilon$ and any allowed $\gamma>0$.
The lower exceptional probability is smaller than every inverse power
of $k$, while the upper one is $O(k^{-n\varepsilon})$.
Since the sampled maximum never exceeds the supremum, the claimed
joint event and the three convergence statements follow.

For (iii), choose $\beta\in(0,1)$ with
$C\beta>\sqrt{1+q/n}$ and then choose
$\sqrt{1+q/n}<\alpha'<C\beta=:\alpha$.
On the event where the supremum is at most
$\alpha\Theta_k^{\max}$ and the sampled maximum is at least
$\beta\Theta_k^{\max}$, one has $R_k\leq C$.
Equations~\eqref{eq:variable-upper} and part~(ii) give
$\Prob_k\{R_k>C\}=O(k^{-q})$.
For every $C>\sqrt{1+1/n}$ one can choose such a $q>1$;
Borel--Cantelli and a countable sequence of $C$ decreasing to the
threshold give the almost-sure assertion.

Finally, radial invariance gives the same ratio probabilities in the
spherical ensemble. For a Gaussian section $g_k$, the section
$g_k/\norm{g_k}_2$ has the spherical distribution and
$\norm{g_k}_2\to1$ in probability, as proved in
Theorem~\ref{thm:sharp}. Dividing the Gaussian amplitude limits by this
norm proves their spherical versions. This also proves
Theorem~\ref{thm:exact-intro}.
\end{proof}

\begin{proposition}[Separation alone need not give exact norming]
\label{prop:variance-counterexample}
If $b_0$ is nonconstant, there are deterministic arrays $\Gamma_k$
satisfying \eqref{eq:general-array} and a number
$0<b'<\max_Xb_0$ such that, in both ensembles,
\[
 \frac{\max_{\Gamma_k}|s|}{\sqrt{b'/V_L}\sqrt{n\log k}}
 \longrightarrow1,\qquad
 R_k(s)\longrightarrow\sqrt{\frac{\max_Xb_0}{b'}}>1
 \quad\text{in probability}.
\]
\end{proposition}
\begin{proof}
Choose a point where $b_0<\max_Xb_0$ and a sufficiently small closed
geodesic ball $K$ about it, of radius less than the injectivity radius,
such that
$b':=\max_Kb_0<\max_Xb_0$. Then $b'>0$, and $K$ is the closure of
its nonempty interior. Fix $\delta>0$ and take $\Gamma_k$ to be a
maximal $\delta/\sqrt{k}$-separated subset of $K$.
Its balls of radius $\delta/\sqrt{k}$ cover $K$.
Volume comparison and packing therefore give
$c_0k^n\leq\#\Gamma_k\leq C_0k^n$, so
\eqref{eq:general-array} holds.

Put $w'=b'/V_L$ and $\Theta'_k=\sqrt{w'n\log k}$.
All node variances are at most $w'+O(k^{-1})$.
For every $\alpha>1$, a union bound gives
\[
 \Prob_k\{\max_{\Gamma_k}|s|>\alpha\Theta'_k\}
 \leq Ck^n\exp\left(-\frac{\alpha^2w'n\log k}{w'+O(k^{-1})}\right)
 \longrightarrow0.
\]
For a matching lower bound, fix $\eta\in(0,1)$.
Continuity and the density of $K^\circ$ in $K$ give an open ball
$W_\eta$ whose closure lies in
$K^\circ\cap\{b_0>(1-\eta)b'\}$.
Choose a smaller concentric ball $W'_\eta$ with closure in $W_\eta$.
For sufficiently large $k$, any node whose covering ball of radius
$\delta/\sqrt{k}$ meets $W'_\eta$ belongs to $W_\eta$.
Since these balls cover $W'_\eta$ and each has volume at most $Ck^{-n}$,
\[
 \#(\Gamma_k\cap W_\eta)\geq c_\eta k^n.
\]
On this subset the variances are at least
$(1-\eta)w'-O(k^{-1})$.
For any fixed $\beta\in(0,1)$ choose $\eta$ with $\beta^2<1-\eta$.
Lemma~\ref{lem:localized-maximum}, with
$w=(1-\eta)w'$ and parameter $\beta/\sqrt{1-\eta}$, shows that
$\Prob_k\{\max_{\Gamma_k}|s|<\beta\Theta'_k\}\to0$.
This proves the sampled-amplitude limit.
The global supremum limit in Theorem~\ref{thm:fekete-exact}(i)
then gives the asserted ratio limit.
Radial normalization transfers the amplitude convergence to the sphere;
the ratio is unchanged by that normalization.
\end{proof}

\begin{corollary}[Dichotomy for separated arrays]\label{cor:variance-dichotomy}
For fixed $(X,\omega,L,h)$, the following are equivalent:
\begin{enumerate}
\item[(i)] $b_0$ is constant, or equivalently
$c_1(L,h)^n=\kappa\omega^n$ for some constant $\kappa>0$.
\item[(ii)] Every deterministic array satisfying
\eqref{eq:general-array} has $R_k\to1$ in probability.
\end{enumerate}
The equivalence holds in both the Gaussian and spherical ensembles.
For every background K\"ahler form, Fekete arrays
$\F_{\lceil ak\rceil}$ are asymptotically exact for each fixed $a>0$.
\end{corollary}
\begin{proof}
Theorem~\ref{thm:sharp} proves (i)$\Rightarrow$(ii).
Proposition~\ref{prop:variance-counterexample} proves the converse by
contraposition. The final assertion is
Theorem~\ref{thm:fekete-exact}.
\end{proof}

The proposition concerns arrays with cardinality comparable to $N_k$,
not necessarily exactly $N_k$ nodes. Together with
Theorem~\ref{thm:fekete-exact}, it identifies the role of distribution
beyond separation. The proof of that theorem applies more generally to
arrays satisfying \eqref{eq:general-array} and
$\#(\Gamma_k\cap U_\eta)\geq c_\eta k^n$ for every fixed $\eta>0$.
Fekete equidistribution guarantees this condition for every background
K\"ahler form.

\subsection{An \texorpdfstring{$L^2$}{L2} comparison by Gaussian quadratic forms}
A deterministic implication from $L^\infty$ sampling to $L^2$ sampling cannot simply be applied to an inequality holding only outside a Gaussian exceptional set, since multiplying a Gaussian section by a peak section changes its distribution.

\begin{lemma}\label{lem:quadratic}
Let $g\in\C^N$ have independent standard complex Gaussian coordinates. Suppose $A$ is Hermitian with
\[
0\leq A\leq M I,\qquad \mu_- N\leq\tr A\leq\mu_+ N,
\]
where $M,\mu_-,\mu_+>0$ do not depend on $N$. There is $c>0$, depending only on these constants, such that
\[
\Prob\left\{N^{-1}g^*Ag\notin[\mu_-/2,2\mu_+]\right\}\leq2e^{-cN}.
\]
\end{lemma}
\begin{proof}
By unitary invariance, write $g^*Ag=\sum_j\theta_jE_j$, where $0\leq\theta_j\leq M$ and the $E_j$ are independent exponential variables of mean one. Put $T=\sum_j\theta_j$. For $|t|\leq(2M)^{-1}$,
\[
\log\E\exp\left(t\sum_j\theta_j(E_j-1)\right)
=\sum_j\{-t\theta_j-\log(1-t\theta_j)\}
\leq2t^2\sum_j\theta_j^2\leq2t^2M\mu_+ N.
\]
Apply the exponential Markov inequality with a sufficiently small fixed positive $t$ to the upper deviation, and with negative $t$ to the lower deviation. The deviations from $T$ are at least $\mu_+ N$ and $\mu_- N/2$, respectively, proving the claim after reducing $c$.
\end{proof}

\begin{theorem}\label{thm:L2}
Let $\Gamma_k$ satisfy \eqref{eq:general-array}. There are constants $A,B,c,C>0$ and measurable sets $E_k\subset H_k$ with
\[
\Prob_k(E_k)\leq Ce^{-ck^n}
\]
such that, for all sufficiently large $k$, the left inequality below holds for every $s\in H_k$, and the right inequality holds for every $s\notin E_k$:
\begin{equation}\label{eq:L2}
A k^{-n}\sum_{\lambda\in\Gamma_k}|s(\lambda)|_{h^k}^2
\leq\norm{s}_2^2\leq
B k^{-n}\sum_{\lambda\in\Gamma_k}|s(\lambda)|_{h^k}^2.
\end{equation}
In particular this holds for $\Gamma_k=\F_k$. Almost every sequence satisfies \eqref{eq:L2} for all sufficiently large $k$. The same probability bound for the failure of \eqref{eq:L2} holds in the spherical ensemble.
\end{theorem}
\begin{proof}
Choose unit frames at the points of $\Gamma_k$, and let $V_k$ be the matrix mapping the coefficient vector of a section to its scalar evaluations. Set
\[
A_k=k^{-n}V_k^*V_k,\qquad Q_k(s)=k^{-n}\sum_{\lambda\in\Gamma_k}|s(\lambda)|^2.
\]
By \eqref{eq:PP}, $0\leq A_k\leq M I$ for a fixed $M>0$. In particular, $Q_k(s)\leq M\norm{s}_2^2$ for every section, so the left inequality in \eqref{eq:L2} holds deterministically with $A=M^{-1}$. Moreover,
\[
\tr A_k=k^{-n}\sum_{\lambda\in\Gamma_k}B_k(\lambda)\asymp k^n\asymp N_k,
\]
using \eqref{eq:diagonal}, the lower cardinality assumption, and the packing upper bound. Thus there are fixed $\mu_-,\mu_+>0$ with $\mu_- N_k\leq\tr A_k\leq\mu_+ N_k$.

Writing $s=N_k^{-1/2}\sum_jg_jS_j^k$, we have
\[
Q_k(s)=N_k^{-1}g^*A_kg,\qquad
\norm{s}_2^2=N_k^{-1}\sum_j|g_j|^2.
\]
Lemma~\ref{lem:quadratic} applied first to $A_k$ and then to the identity shows that, except on a set of probability at most $Ce^{-cN_k}$,
\[
\mu_-/2\leq Q_k(s)\leq2\mu_+,\qquad
1/2\leq\norm{s}_2^2\leq2.
\]
On this event the right inequality in \eqref{eq:L2} holds with $B=4/\mu_-$. Together with the deterministic left inequality, this proves \eqref{eq:L2}. Define $E_k$ to be the failure event for the right inequality; its probability is bounded by that of the preceding union. Since $N_k\asymp k^n$, this proves the estimate. Borel--Cantelli gives the almost-sure assertion. Finally, $E_k$ is invariant under positive radial scaling, so the Gaussian and spherical probabilities agree.
\end{proof}

\begin{remark}
Theorem~\ref{thm:L2} does not require the evaluation matrix to have a uniform lower singular-value bound. The bounded operator norm and trace of order $N_k$ suffice for the random quadratic-form estimate. This explains why the result also applies to arrays with fewer points than $N_k$, even though their evaluation maps have nontrivial kernels.
\end{remark}

\end{document}